\documentclass{GANITA}
\citesort

\theoremstyle{cupthm}
\newtheorem{thm}{Theorem}[section]
\newtheorem{prop}[thm]{Proposition}
\newtheorem{cor}[thm]{Corollary}

\theoremstyle{cupdefn}
\newtheorem{defn}[thm]{Definition}
\theoremstyle{cuprem}
\newtheorem{rem}[thm]{Remark}
\numberwithin{equation}{section}

\begin{document}
\runningtitle{Quasi-isometry between two almost contact metric manifolds}
\title{Quasi-isometry between two almost contact metric manifolds}
%% If there is more than one author, put \cauthor immediately before
%% the corresponding author.
%\cauthor %% mark the next author as corresponding author
\author[1]{Arindam Bhattacharyya}
\address[1]{Department of Mathematics, Jadavpur University,
Kolkata-700032, India
\email{arindam.bhattacharyya@jadavpuruniversity.in}}
\author[2]{Dipen Ganguly}
\address[2]{Department of Mathematics, Jadavpur University,
Kolkata-700032, India
\email{dipenganguly1@gmail.com}}
\author[3]{Paritosh Ghosh}
\address[3]{Department of Mathematics, Jadavpur University,
Kolkata-700032, India
\email{paritoshghosh112@gmail.com}}
\author[3]{Sumanjit Sarkar}
\address[3]{Department of Mathematics and Statistics,
Vignan's Foundation for Science, Technology and Research; Guntur, Andhra Pradesh-522213, India 
\email{imsumanjit@gmail.com}}
%% If there are several authors, list them here
%\author[2]{Second author}
%\address[2]{Second address\email{a@net.com}}

%% List the authors, initials and surnames only, for the
%% running head (left hand page)
\authorheadline{A. Bhattacharyya et al}

%% If there is a dedication, include it here
%\dedication{Dedicated to ...}

\support{Author P. Ghosh is financially supported by UGC Junior Research Fellowship of India (Ref. No: 201610010610).\\
Author A. Bhattacharyya is the corresponding author.}

\begin{abstract}
In this paper the notion of quasi-isometry between two Riemannian manifolds has been introduced. This idea is also imposed to study quasi-isometry between two almost contact metric manifolds. Moving further, some curvature properties of two quasi-isometrically embedded almost contact metric manifolds, $N(k)-$contact metric manifolds and Sasakian manifolds are investigated. Next, an illustrative example of a quasi-isometry between two Sasakian structures is constructed. Finally, a relation between the scalar curvature and the quasi-isometric constants for two quasi-isometric Riemannian manifolds has been established.
\end{abstract}

%% - subject classification and keywords
%% 2010 American Mathematical Society Subject Classification
%% Provide only ONE primary classification
\classification{primary 53C20; secondary 53C25}
%% Four or five keywords or phrases
\keywords{Quasi-isometry, Almost contact metric manifold, $N(k)-$ Contact metric manifold, Einstein manifold, Sasakian Manifold}

\maketitle

\section{Introduction}
\label{intro}
The notion quasi-isometry was first introduced by the American mathematician G.D. Mostow \cite{Mos} in 1973 and later it was Gromov \cite{Gr1} who studied quasi-isometry to a much further extent in the context of geometric group theory. But Mostow used the term pseudo-isometry and this notion was little bit different from the one that we will be discussing here (See \cite{BH}). Let $(X,d_x)$ and $(Y,d_y)$ be two metric spaces and $f:(X,d_x)\longrightarrow(Y,d_y)$ be a map. Then the map $f$ is said to be an $(L,C)$ quasi-isometric embedding, if there exist constants $L\geq 1$ and $C\geq 0$ such that for all $p,q\in(X,d_x)$
\begin{equation}\label{e1.1}
\frac{1}{L}d_x(p,q)-C\leq d_y(f(p),f(q))\leq Ld_x(p,q)+C.
\end{equation}
Moreover, if the quasi-isometric embedding $f$ has a quasi dense image, i.e if there is a constant $D\geq 0$ such that $\forall y\in Y$, $\exists x\in X$ for which $d_y(f(x),y)\leq D$, then the map $f$ is called a quasi isometry and we call that the two metric spaces  $(X,d_x)$ and $(Y,d_y)$ are quasi-isometric. For example, it can shown that the grid $\mathbb{Z}^2$ with the taxicab metric is quasi-isometric to the plane $\mathbb{R}^2$ with the usual Euclidean metric via the natural inclusion map as a Quasi-isometry \cite{CM}. Also it is easy to see that any metric space of finite diameter is quasi-isometric to a point. In that manner we can say that, all metric spaces of finite diameter are same in the sense of quasi-isometry.\par

We say that a map $f:X\longrightarrow Y$ has finite distance from a map $g:X\longrightarrow Y$ if there is a constant $M\geq 0$ such that, for all $x\in X$ we have $d_x(g(x),f(x))\leq M$ and $f\sim g$ if $f$ is at finite distance from $g$. Then it is easy to check that $'\sim '$ is an equivalence relation. We denote $QI(X)$ be the set of all quasi-isometries from $X\longrightarrow X$, and let $QI(X)/\sim$ be the set of all quasi isometries of X modulo finite distance. Moreover the composition $([f],[g])\mapsto[f\circ g]$ on the set of all  equivalence classes $QI(X)$ forms a group, called the quasi-isometry group of X. It is a major problem in geometric group theory, to find the quasi-isometry groups of spaces.\par

In geometric group theory the main idea is to see how groups can be viewed as geometric objects. To be more precise on which geometric object can act as a group in a 'nice way' so that the interplay between the group and the space reveals the algebraic properties of the group. In this direction one fundamental result is the $\check{S}$varc-Milnor lemma, which says that if a group G acts properly and co-compactly by isometries on a non-empty proper geodesic metric space $(X,d)$, then G is finitely generated and for all $x\in X$ the map $$G\longrightarrow X$$ $$g\mapsto g.x$$ is a quasi-isometry (a metric space is proper if all balls of finite radius are compact in the metric topology and an action of a group G on a topological space X is co-compact if the quotient space $X/G$ is compact with respect to the quotient topology).\par

One of the central theorems in geometric group theory is Gromov's polynomial growth theorem (See \cite{Gr1}), which says that finitely generated groups have polynomial growth if and only if they are virtually nilpotent (i.e if the group has a subgroup of finite index that is nilpotent). Then using this theorem an interesting result can be proved that, if a group G is quasi-isometric to $\mathbb{Z}^n$ then G has a subgroup of finite index which is isomorphic to $\mathbb{Z}^n$.\par

In Riemannian Geometry, two Riemannian manifolds $(M_1^{m_1},g_1)$ and $(M_2^{m_2},g_2)$ are said to be isometric if there exists a  diffeomorphism $f:M_1\to M_2$ such that $g_2(f_*X,f_*Y)=g_1(X,Y)$ for all $X,Y\in\chi(M_1)$, where $f_*:\chi(M_1)\to\chi(M_2)$ is the differential of $f$, where $\chi(M_1)$ and $\chi(M_2)$ are set of all vector fields of $M_1$ and $M_2$ respectively. Such a map $f$ is called isometry. This motivates us to define quasi-isometry between two Riemannian manifolds.
\begin{defn}\label{def1}
Let $(M_1^{m_1},g_1)$ and $(M_2^{m_2},g_2)$ be two Riemannioan manifolds of respective dimensions $m_1$ and $m_2$. Let $\chi(M_1)$ and $\chi(M_2)$ be the set of all vector fields associated to $M_1$ and $M_2$ respectively. A diffeomorphism $f:M_1^{m_1}\to M_2^{m_2}$ is said to be a quasi-isometric embedding between $M_1$ and $M_2$ if there exist constants $A\geq1, B\geq0$ such that for all $X,Y\in\chi(M_1)$,
\begin{equation}\label{e1.2}
  \frac{1}{A}g_1(X,Y)-B \leq g_2(f_*(X),f_*(Y))\leq Ag_1(X,Y)+B.
\end{equation}
Moreover, if for all $Z\in\chi(M_2)$ there exists $X\in\chi(M_1)$ and a constant $D\geq0$ such that
\begin{equation}\label{e1.3}
g_2(Z,f_*(X))\leq D,
\end{equation}
then $f$ is called quasi-isometry between the manifolds $M_1$ and $M_2$.\\
The two manifolds $M_1$ and $M_2$ are called quasi-isometric if there exists such a quasi-isometry $f$ between $M_1$ and $M_2$.
\end{defn}
The definition given in \eqref{e1.1} is based on usual metric of the metric space whereas in Definition \ref{def1} we have considered the Riemannian metric for the inequalities, which is more generalized form than the usual metric $d$.\par
In this paper we have introduced the concept of quasi-isometry for almost contact metric manifolds, for $N(k)-$contact metric manifolds and for Sasakian manifolds of same dimensions and established some inequalities between two quasi-isometric metric manifolds for various cases like when the ambient manifold is conformally flat, concircularly flat, etc. We have given an example of a quasi-isometry between two Sasakian manifolds. And finally we find a relationship between the scalar curvature and the quasi-isometric constants for two Riemannian manifolds to be quasi-isometric.

\section{Preliminaries}\label{sec2}

A \emph{contact manifold} $M^{2n+1}$ is a $C^\infty$ manifold together with a global 1-form $\eta$ such that $\eta \wedge (d\eta)^n \neq 0$. More specifically, $\eta \wedge (d\eta)^n $ is a volume element on M, which is non-zero everywhere on $M^{2n+1}$ so that the manifold M is orientable.\par

Let $M^{2n+1}$ be a $(2n+1)$ dimensional manifold and let there exist a $(1,1)$ tensor field $\phi$, a vector field $\xi$ and a global 1-form $\eta$ on M such that
\begin{eqnarray}
   \phi ^2 &=& -I + \eta \otimes \xi,\label{e2.1}\\
  \eta (\xi ) &=& 1\label{e2.2},
\end{eqnarray}
then we say that M has an \emph{almost contact structure} $(\phi,\xi,\eta)$. And the manifold M equipped with this almost contact structure $(\phi,\xi,\eta)$ is called an \emph{almost contact manifold} (See \cite{Bla}). \par
Here the vector field $\xi$ is called the \emph{characteristic vector field} or \emph{Reeb vector field}.\par

\begin{prop}
 \cite{Bla} For an almost contact structure $(\phi,\xi,\eta)$ the following relations hold:
\begin{eqnarray}
  \phi\circ\xi &=& 0,\label{e2.3}\\
  \eta\circ\phi &=& 0, \label{e2.4}\\
  Rank\phi &=& 2n\label{e2.5}.
\end{eqnarray}
\end{prop}

\begin{thm}
\cite{Bla} Every almost contact structure $(\phi,\xi,\eta)$ on a manifold $M^{2n+1}$ admits a Riemannian metric $g$ satisfying:
  \begin{eqnarray}
    \eta(X) &=& g(X,\xi),\label{e2.6}\\
    g(\phi X,\phi Y) &=& g(X,Y)-\eta(X)\eta(Y)\label{e2.7}.
  \end{eqnarray}
\end{thm}
And the metric $g$ is called \emph{compatible} with the almost contact structure $(\phi,\xi,\eta)$ and the manifold $M^{2n+1}$ with the almost contact metric structure $(\phi,\xi,\eta,g)$ is called an \emph{almost contact metric manifold}.\par

In 1988, S. Tanno \cite{ST} introduced the notion of $k-$nullity distribution on a contact metric manifold which is defined as follows: The $k-$nullity distribution of a Riemannian manifold $(M,g)$ for a real number k is a distribution,
\begin{equation}
N(k): p\longrightarrow N_p(k)=[Z\in T_pM: R(X, Y)Z=k\{g(Y, Z)X- g(X, Z)Y\}],\label{e2.8}
\end{equation}
for any $X, Y\in T_pM$, where $R$ is the \emph{Riemannian curvature tensor} and $T_pM$ denotes the tangent vector space of $M^{2n+1}$ at point $p\in M$.\par \medskip
If the characteristic vector field of a contact metric manifold belongs to the $k-$nullity distribution, then the relation,
\begin{equation}\label{e2.9}
R(X, Y)\xi=k[\eta(Y)X-\eta(X)Y]
\end{equation}
holds. A contact metric manifold with $\xi \in N(k)$ is called a \emph{$N(k)-$contact metric manifold}.\par

\begin{prop}
\cite{Bla} Let $M^{2n+1}(\phi,\xi,\eta, g) (n\geq2)$ be a $N(k)-$contact metric manifold. Then the following relations hold:
\begin{eqnarray}
   Q\xi &=& (2nk)\xi, \label{e2.10}\\
   S(X, \xi) &=& 2nk\eta(X), \label{e2.11}\\
   \eta(R(X,Y)Z) &=& k[\eta(X)g(Y,Z)-\eta(Y)g(X,Z)],\label{e2.12}
\end{eqnarray}
where, $R$ is the Riemannian curvature tensor, $S$ is the Ricci tensor of type $(0,2)$ and $Q$ is the Ricci operator or the symmetric endomorphism of the tangent space $T_pM$ at the point $p\in M$ and is given by $S(X, Y)=g(QX, Y)$.
\end{prop}

Next we recall a very important manifold named Sasakian manifold which was introduced by the Japanese mathematician S. Sasaki \cite{Sa} in the year 1960. Later, the works of Boyer, Galicki \cite{BG} and other mathematicians have made a substantial progress in the study of Sasakian manifolds. In mathematical physics Sasakian manifolds and more specifically Sasakian space forms are widely used. Sasakian manifolds or normal contact metric manifolds are an odd-dimensional counterpart of the K\"{a}hler manifolds in complex geometry.\par
An almost contact manifold $M^{2n+1}$ together with the almost contact structure $(\phi,\xi,\eta)$ is said to be a \emph{Sasakian manifold} or a \emph{normal contact metric manifold} if
\begin{equation*}
[\phi,\phi](X,Y)+2d\eta(X,Y)\xi=0,
\end{equation*}
where, $[\phi,\phi]$ is the \emph{Nijenhuis torsion tensor field} of $\phi$ and is given by,
\begin{equation*}
[\phi,\phi](X,Y)=\phi^2[X,Y]+[\phi X,\phi Y]-\phi([\phi X,Y])-\phi([X,\phi Y]).
\end{equation*}
\begin{thm}
An almost contact metric manifold $M^{2n+1}$ with the structure $(\phi,\xi,\eta,g)$ is Sasakian if and only if
\begin{equation*}
(\nabla_{X}\phi)Y=g(X,Y)\xi -\eta(Y)X,
\end{equation*}
where, $\nabla$ is the Levi-Civita connection on $M^{2n+1}$ (See \cite{Bla}).
\end{thm}
\begin{prop}
\cite{Bla} Let $M^{2n+1}$ be a Sasakian manifold with the structure $(\phi,\xi,\eta,g)$, then the following relations are true:
\begin{eqnarray}
  \nabla_X\xi &=& -\phi X, \label{e2.13}\\
  R(X,Y)\xi &=& \eta(Y)X-\eta(X)Y, \label{e2.14}\\
  R(X,\xi)Y &=& \eta(Y)X-g(X,Y)\xi, \label{e2.15}\\
  \eta(R(X,Y)Z) &=& \eta(X)g(Y,Z)-\eta(Y)g(X,Z),\label{e2.16}
\end{eqnarray}
where, $R$ is the Riemannian curvature tensor of $M^{2n+1}$ and is given by,
\begin{equation*}
R(X,Y)Z=\nabla_X\nabla_Y Z-\nabla_Y\nabla_X Z-\nabla_{[X,Y]}Z,
\end{equation*}
for all vector fields X,Y,Z on M.
\end{prop}
The theorems and results that are stated above will be used frequently in proofs of the next chapters. For a detailed discussion and proofs of these we refer to the text \cite{Bla}.

\section{Quasi-isometry between two almost contact metric manifolds}
Consider two odd dimensional almost contact metric manifolds $M_1$ and $M_2$ with the structure $(\phi_1,\xi_1,\eta_1,g_1)$ and $(\phi_1,\xi_1,\eta_1,g_1)$ respectively. In this section, we study quasi-isometry between two almost contact metric manifolds $M_1$ and $M_2$.
 
As in \eqref{e1.2}, for all $X, Y \in \chi(M_1)$, we have,
\begin{equation*}
 \qquad \frac{1}{A}g_1(X,Y)-B\leq g_2(f_*(X),f_*(Y))\leq Ag_1(X,Y)+B.
\end{equation*}
For $Y=\xi_1$, we get using \eqref{e2.6},
\begin{eqnarray*}
  \frac{1}{A}\eta_1(X)-B\leq &g_2(f_*(X), f_*(\xi_1))&\leq A\eta_1(X)+B.
\end{eqnarray*}
\par If the function $f_*$ preserves the structure vector field between the two manifolds $M_1$ and $M_2$, that is, if $f_*(\xi_1)=\xi_2$, then $$g_2(f_*(X), f_*(\xi_1))=g_2(f_*(X), \xi_2)=\eta_2(f_*(X)),$$ so that,
\begin{equation}\label{e3.1}
\frac{1}{A}\eta_1(X)-B\leq \eta_2(f_*(X))\leq A\eta_1(X)+B, \quad \forall X\in \chi(M_1).
\end{equation}
\par Since the tensor field $\phi$ is anti-symmetric with respect to the Riemannian metric $g$, that is,
\begin{equation*}
  g(\phi X, Y) = -g(X, \phi Y),
\end{equation*}
we have,
\begin{equation*}
  \quad g_1(\phi_1 X, X) = 0.
\end{equation*}
So, replacing $X$ by $\phi_1X$, we get from \eqref{e1.2},
  $$\frac{1}{A}g_1(\phi_1X,Y)-B\leq g_2(f_*(\phi_1X),f_*(Y))\leq Ag_1(\phi_1X,Y)+B.$$
Again for $Y=X$,
\begin{equation}\label{e3.2}
-B\leq g_2(f_*(\phi_1X), f_*(X))\leq B, \quad \forall X\in \chi(M_1).
\end{equation}
Now replacing $X$ by $\phi_1X$ and $Y$ by $\phi_1Y$ we get from \eqref{e1.2},
$$\frac{1}{A}g_1(\phi_1X,\phi_1Y)-B\leq g_2(f_*(\phi_1X),f_*(\phi_1Y))\leq Ag_1(\phi_1X,\phi_1Y)+B.$$
The left inequality of the above implies
\begin{equation}\label{e3.3}
\frac{1}{A}g_1(X,Y)-B\leq g_2(f_*(\phi_1X),f_*(\phi_1Y))+\frac{1}{A}\eta_1(X)\eta_1(Y).
\end{equation}
Similarly, the right inequality gives
\begin{equation}\label{e3.4}
 g_2(f_*(\phi_1X),f_*(\phi_1Y))+A\eta_1(X)\eta_1(Y)\leq Ag_1(X,Y)+B.
\end{equation}
Since $A\geq1$, we have
\begin{equation}\label{e3.5}
 A\geq \frac{1}{A}.
\end{equation}
Using \eqref{e3.5}, from \eqref{e3.3} and \eqref{e3.4}, we get for all $X, Y\in \chi(M_1)$,
\begin{equation}\label{e3.6}
\frac{1}{A}g_1(X,Y)-B\leq g_2(f_*(\phi_1X),f_*(\phi_1Y))+\frac{1}{A}\eta_1(X)\eta_1(Y)\leq Ag_1(X, Y)+B.
\end{equation}
Replacing $X$ by $\phi_1X$ and $Y$ by $\phi_1Y$, the above implies
\begin{equation*}
\frac{1}{A}g_1(X,Y)-B\leq g_2(f_*({\phi_1}^2X),f_*({\phi_1}^2Y))+\frac{1}{A}\eta_1(X)\eta_1(Y)\leq Ag_1(X, Y)+B.
\end{equation*}
\par Now using the linearity of the differential $f_*$ and \eqref{e2.1} and also if $f_*$ preserves the structure vector field between the two manifolds, a simple calculation leads to
\begin{eqnarray*}
  \frac{1}{A}g_1(X,Y)-B \leq g_2(f_*(X),f_*(Y))-\eta_1(X)g_2(\xi_2,f_*(Y))\\
  -\eta_1(Y)g_2(f_*(X),\xi_2)+\eta_1(X)\eta_1(Y)(g_2(\xi_2,\xi_2)+\frac{1}{A})\leq Ag_1(X, Y)+B,
\end{eqnarray*}
which implies
\begin{eqnarray}\label{e3.7}
  \frac{1}{A}g_1(X,Y)-B \leq g_2(f_*(X),f_*(Y))-\eta_1(X)\eta_2(f_*(Y)) \nonumber \\
  -\eta_1(Y)\eta_2(f_*(X))+\eta_1(X)\eta_1(Y)(1+\frac{1}{A})\leq Ag_1(X, Y)+B.
\end{eqnarray}
So collecting all these results, we can state that:
\begin{thm}
Let $M_1(\phi_1,\xi_1,\eta_1,g_1)$ and $M_2(\phi_2,\xi_2,\eta_2,g_2)$ be two odd dimensional almost contact metric manifolds and let $f:M_1\to M_2$ be a quasi-isometric embedding. Also consider that $f_*$ preserves the structure vector field between the two manifolds. Then for all $X, Y\in\chi(M_1)$, the following relations hold:
\begin{enumerate}
  \item $\frac{1}{A}\eta_1(X)-B \leq \eta_2(f_*(X)) \leq A\eta_1(X)+B,$
  \item $-B \leq g_2(f_*(\phi_1X), f_*(X)) \leq B,$
  \item $\frac{1}{A}g_1(X,Y)-B \leq g_2(f_*(\phi_1X),f_*(\phi_1Y))+\frac{1}{A}\eta_1(X)\eta_1(Y) \leq Ag_1(X, Y)+B,$
  \item $\frac{1}{A}g_1(X,Y)-B \leq g_2(f_*(X),f_*(Y))-\eta_1(X)\eta_2(f_*(Y))-\eta_1(Y)\eta_2(f_*(X))+\eta_1(X)\eta_1(Y)(1+\frac{1}{A})\leq Ag_1(X, Y)+B.$
\end{enumerate}
\end{thm}

\section{Quasi-isometry between two $N(k)-$contact metric manifolds}
In this section we deal with the quasi-isometry between two $N(k)-$ contact metric manifolds in a similar way and establish some interesting results.
 \par Recall that if a transformation does not change the angle between the tangent vectors of a manifold, it is called a conformal transformation. The Weyl conformal curvature tensor C of a Riemannian manifold $(M,g)$ of dimension $2n+1$ $(n\geq1)$ is an invariant under any conformal transformation of the metric $g$ and it is defined by
\begin{eqnarray}\label{e4.1}
 &C(X,Y)Z=R(X,Y)Z-\frac{1}{(2n-1)}[S(Y,Z)X-S(X,Z)Y+g(Y,Z)QX\nonumber\\
 &-g(X,Z)QY]+\frac{r}{2n(2n-1)}[g(Y,Z)X-g(X,Z)Y],
\end{eqnarray}
 where $R$ is \emph{ Riemannian curvature tensor}, $S$ is the \emph{Ricci tensor of type $(0,2)$}, $Q$ is the \emph{Ricci operator} given by $S(X,Y)=g(QX,Y)$, and $r$ is the \emph{scalar curvature} of the manifold $M$.\par
Next, let the manifold $M_1$ be \emph{conformally flat} i.e; $C_1(X,Y)Z=0$ for all $X,Y,Z\in\chi(M_1)$. Then from the equation \eqref{e4.1} we get,
\begin{eqnarray}\label{e4.2}
R_1(X,Y)Z&=\frac{1}{(2n-1)}[S_1(Y,Z)X-S_1(X,Z)Y+g_1(Y,Z)Q_1X\nonumber\\
&-g_1(X,Z)Q_1Y]-\frac{r}{2n(2n-1)}[g_1(Y,Z)X-g_1(X,Z)Y].
\end{eqnarray}
Putting $Z=\xi_1$ and using \eqref{e2.8},\eqref{e2.9} and the relation $S_1(X, \xi_1)=2nk\eta_1(X)$, we get after some calculations
\begin{equation}\label{e4.3}
R_1(X,Y)\xi_1=\frac{2nk}{r-2nk}[\eta_1(Y)Q_1X-\eta_1(X)Q_1Y].
\end{equation}
And for $Y=\xi_1$,
\begin{equation}\label{e4.4}
Q_1X=(\frac{r-2nk}{2n})X+[(2n+1)k-\frac{r}{2n}]\eta_1(X)\xi_1.
\end{equation}
Now since $R_1(X,Y)Z\in\chi(M_1)$, for all $X,Y,Z,W\in\chi(M_1)$, the left side inequality of \eqref{e1.2} implies
\begin{equation}\label{e4.5}
\frac{1}{A}g_1(R_1(X,Y)Z,W)-B\leq g_2(f_*(R_1(X,Y)Z),f_*(W)).
\end{equation}
Using \eqref{e4.2}, the above inequality becomes
\begin{eqnarray}\label{e4.6}
 &\frac{1}{A}[\frac{1}{(2n-1)}\{S_1(Y,Z)g_1(X,W)-S_1(X,Z)g_1(Y,W)+\nonumber\\
 &g_1(Y,Z)g_1(Q_1X,W)-g_1(X,Z)g_1(Q_1Y,W)\}-\frac{r}{2n(2n-1)}\{g_1(Y,Z)\nonumber\\
 &g_1(X,W)-g_1(X,Z)g_1(Y,W)\}]-B \leq g_2(f_*(R_1(X,Y)Z),f_*(W)).
\end{eqnarray}
For $Z=\xi_1$, the above gives
\begin{eqnarray}\label{e4.7}
&\frac{1}{A}[\frac{1}{(2n-1)}\{2nk\eta_1(Y)g_1(X,W)-2nk\eta_1(X)g_1(Y,W)\nonumber\\
&+\eta_1(Y)S_1(X,W)-\eta_1(X),S_1(Y,W)\}-\frac{r}{2n(2n-1)}\{\eta_1(Y)\nonumber\\
&g_1(X,W)-\eta_1(X)g_1(Y,W)\}]-B\leq g_2(f_*(R_1(X,Y)\xi_1),f_*(W)).
\end{eqnarray}
Simplifying after some steps and assuming $[\frac{1}{(2n-1)}(2nk-\frac{r}{2n})]=l_1$ and $\frac{1}{(2n-1)}=l_2$ we get
\begin{eqnarray}\label{e4.8}
&\frac{1}{A}[l_1\{\eta_1(Y)g_1(X,W)-\eta_1(X)g_1(Y,W)\}+l_2\{\eta_1(Y)S_1(X,W)\nonumber\\
&-\eta_1(X)S_1(Y,W)\}]-B\leq g_2(f_*(R_1(X,Y)\xi_1),f_*(W)).
\end{eqnarray}
Setting $Y=Z=\xi_1$ in \eqref{e4.2}, it can be shown that conformally flat $N(k)-$ contact metric manifold $M_1$ becomes $\eta$-Einstein manifold, that is, $$S_1(X,Y)=ag_1(X,Y)+b\eta_1(X)\eta_1(Y),$$ where, $a=[\frac{r}{2n}-k]$ and $b=[(2n+1)k-\frac{r}{2n}]$. Then putting this value of $S_1$ in \eqref{e4.8} and after simplification we have
\begin{eqnarray}\label{e4.9}
&\frac{1}{A}[l_1\{\eta_1(Y)g_1(X,W)-\eta_1(X)g_1(Y,W)\}+l_2\{\eta_1(Y)(ag_1(X,W)\nonumber\\
&+b\eta_1(X)\eta_1(W))-\eta_1(X)(ag_1(Y,W)+b\eta_1(Y)\eta_1(W))\}]\nonumber\\
&-B\leq g_2(f_*(R_1(X,Y)\xi_1),f_*(W)).
\end{eqnarray}
Now, using \eqref{e2.12} and observing that $(l_1+al_2)=k$, the above inequality becomes
\begin{equation}\label{e4.10}
\frac{1}{A}\eta_1(R_1(Y,X)W)-B\leq g_2(f_*(R_1(X,Y)\xi_1),f_*(W)).
\end{equation}
Reminding the linearity of $f_*$ and using the relation \eqref{e2.9}, the last inequality leads to
\begin{equation}\label{e4.11}
\frac{1}{A}\eta_1(R_1(Y,X)W)-B\leq k[\eta_1(Y)g_2(f_*(X),f_*(W))-\eta_1(X)g_2(f_*(Y),f_*(W))].
\end{equation}
Similarly, taking the right side inequality of the \eqref{e1.2} and proceeding as above we get 
\begin{equation}\label{e4.12}
k[\eta_1(Y)g_2(f_*(X),f_*(W))-\eta_1(X)g_2(f_*(Y),f_*(W))]\leq A\eta_1(R_1(Y,X)W)]+B.
\end{equation}
So, combining the inequalities \eqref{e4.11} and \eqref{e4.12}, we can write
\begin{thm}
Let $M_1(\phi_1,\xi_1,\eta_1,g_1)$ and $M_2(\phi_2,\xi_2,\eta_2,g_2)$ be two odd dimensional $N(k)-$contact metric manifolds with $\operatorname{dim}M_1=2n+1$ $(n>1)$. Suppose $f:M_1\to M_2$ be a quasi-isometric embedding with the constants $A\geq1, B\geq0$. Furthermore, if the manifold $M_1$ is conformally flat, then for all $X,Y,W\in\chi(M_1)$, the metric $g_2$ of the manifold $M_2$ satisfies
\begin{eqnarray}\label{e4.13}
\frac{1}{A}\eta_1(R_1(Y,X)W)-B\leq k[\eta_1(Y)g_2(f_*(X),f_*(W))\nonumber\\
-\eta_1(X)g_2(f_*(Y),f_*(W))]\leq A\eta_1(R_1(Y,X)W)]+B,
\end{eqnarray}
where $R_1$ is the Riemannian curvature tensor of the manifold $M_1$.
\end{thm}
\begin{rem}
Now consider $f$ is a quasi-isometry between $M_1$ and $M_2$. Also consider $f_*(X)=Z_1$ and $f_*(Y)=Z_2$. Then there exists some $W \in \chi(M_1)$ such that $g_2(Z_1, f_*(W))\leq D$ and $g_2(Z_2, f_*(W))\leq D$, where $D\geq0$. So from \eqref{e4.11}, we get
\begin{equation*}
  \frac{1}{A}\eta_1(R_1(Y,X)W)-B\leq kD\eta_1(Y-X).
\end{equation*}
After a small calculation we can remark that,
\begin{equation*}
  R_1(Y,X)W\leq A(B\xi_1+kD(Y-X)).
\end{equation*}
\end{rem}
The following corollary can also be demonstrated:
\begin{cor}
Let $f$ be a quasi-isometric embedding between two $N(k)-$ contact metric manifolds $M_1(\phi_1,\xi_1,\eta_1,g_1)$ and $M_2(\phi_2,\xi_2,\eta_2,g_2)$ with $M_1$ conformally flat. If $f_*$ preserves the structure vector field, then for some $A\geq1$ and $B_1\geq0$ we have
\begin{equation}\label{e4.14}
-B_1\leq \eta_1(Y)g_2(f_*(X),\xi_2)-\eta_1(X)g_2(f_*(Y),\xi_2)\leq B_1.
\end{equation}
\end{cor}
\begin{proof}
The proof of this corollary follows from the equation \eqref{e4.13} after putting $W=\xi_1$ and using \eqref{e2.9}.
\end{proof}\par

Next, we consider the manifold $M_1$ to be \emph{conformally flat Einstein manifold}, then its Ricci tensor $S_1$ satisfies  $S_1(X,Y)=\frac{r}{2n+1}g_1(X,Y)$. Now using this in \eqref{e4.8} we get
\begin{eqnarray}\label{e4.15}
\frac{1}{A}[l_1\{\eta_1(Y)g_1(X,W)-\eta_1(X)g_1(Y,W)\}+l_2\frac{r}{2n+1}\{\eta_1(Y)g_1(X,W)\nonumber\\
-\eta_1(X)g_1(Y,W)\}]-B\leq g_2(f_*(R_1(X,Y)\xi_1),f_*(W)).
\end{eqnarray}
Then after simplification this yields
\begin{equation*}
\frac{1}{Ak}(l_1+\frac{r}{2n+1}l_2)\eta_1(R_1(Y,X)W)-B\leq g_2(f_*(R_1(X,Y)\xi_1),f_*(W)).
\end{equation*}
Now considering $a_1=\frac{1}{k}(l_1+\frac{r}{2n+1}l_2)=\frac{1}{k(2n-1)}(2nk-\frac{r}{2n}+\frac{r}{2n+1})$, the above inequality transforms into
\begin{equation*}
\frac{a_1}{A}\eta_1((R_1(Y,X)W))-B\leq g_2(f_*(R_1(X,Y)\xi_1),f_*(W)).
\end{equation*}
Applying the linearity of $f_*$ the last inequality becomes
\begin{eqnarray}\label{e4.16}
&\frac{a_1}{A}\eta_1((R_1(Y,X)W))-B\leq k[\eta_1(Y)g_2(f_*(X),f_*(W))\nonumber\\
&-\eta_1(X)g_2(f_*(Y),f_*(W))].
\end{eqnarray}
Again proceeding similarly with the right side inequality, we have
\begin{eqnarray}\label{e4.17}
&k[\eta_1(Y)g_2(f_*(X),f_*(W))-\eta_1(X)g_2(f_*(Y),f_*(W))]\nonumber\\
&\leq a_1A\eta_1((R_1(Y,X)W))+B.
\end{eqnarray}
Hence, from \eqref{e4.16} and \eqref{e4.17}, we can state the following corollary
\begin{cor}
Let $f$ be a quasi-isometric embedding between two $N(k)-$ contact metric manifolds $M_1(\phi_1,\xi_1,\eta_1,g_1)$ and $M_2(\phi_2,\xi_2,\eta_2,g_2)$. Moreover, if the manifold $M_1$ be conformally flat Einstein manifold, then for some $A\geq1, B\geq0$ and for all $X,Y,W\in\chi(M_1)$ the metric $g_2$ of $M_2$ satisfies
\begin{eqnarray}\label{e4.18}
&\frac{a_1}{A}\eta_1((R_1(Y,X)W))-B\leq k[\eta_1(Y)g_2(f_*(X),f_*(W))\nonumber\\
&-\eta_1(X)g_2(f_*(Y),f_*(W))]\leq a_1A\eta_1((R_1(Y,X)W))+B,
\end{eqnarray}
where $R_1$ is the Riemannian curvature tensor of the manifold $M_1$ and $a_1=\frac{1}{k(2n-1)}(2nk-\frac{r}{2n}+\frac{r}{2n+1})$.
\end{cor}\par

The concircular curvature tensor of a manifold $(M^{2n+1},g)$ is given by
\begin{equation*}
\bar{C}(X,Y)Z=R(X,Y)Z-\frac{r}{2n(2n+1)}[g(Y,Z)X-g(X,Z)Y].
\end{equation*}
Now, if our ambient manifold $M_1$ be \emph{concircularly flat} i.e $\bar{C}(X,Y)Z=0$, then from above we have $$R_1(X,Y)Z=\frac{r}{2n(2n+1)}[g(Y,Z)X-g(X,Z)Y].$$ Putting this value in the left inequality of \eqref{e1.2}, we get
\begin{eqnarray*}
&\frac{r}{2An(2n+1)}[g_1(Y,Z)g_1(X,W)-g_1(X,Z)g_1(Y,W)]\\
&-B\leq g_2(f_*(R_1(X,Y)Z),f_*(W)).
\end{eqnarray*}
Then for $Z=\xi_1$ and using \eqref{e2.6}, \eqref{e2.12} with the linearity of $f_*$ and simplifying we have
\begin{equation}\label{e4.19}
\frac{b_1}{A}\eta_1(R_1(Y,X)W)-B\leq k[\eta_1(Y)g_2(f_*(X),f_*(W))-\eta_1(X)g_2(f_*(Y),f_*(W))],
\end{equation}
with $b_1=\frac{r}{2nk(2n+1)}$.\\
Again proceeding similarly as above, from the right side inequality of \eqref{e1.2}, we have
\begin{equation}\label{e4.20}
k[\eta_1(Y)g_2(f_*(X),f_*(W))-\eta_1(X)g_2(f_*(Y),f_*(W))]\leq b_1A\eta_1(R_1(Y,X)W)+B.
\end{equation}
So, combining \eqref{e4.19} and \eqref{e4.20} we get
\begin{thm}
Let $M_1(\phi_1,\xi_1,\eta_1,g_1)$ and $M_2(\phi_2,\xi_2,\eta_2,g_2)$ be two $N(k)-$ contact metric manifolds with $\operatorname{dim}M_1=2n+1$ $(n\geq1)$. Suppose $f:M_1\to M_2$ be a quasi-isometric embedding with the constants $A\geq1, B\geq0$. Moreover, if $M_1$ is concircularly flat, then for all $X,Y,W\in\chi(M_1)$ we have
\begin{eqnarray}\label{e4.21}
&\frac{b_1}{A}\eta_1(R_1(Y,X)W)-B\leq k[\eta_1(Y)g_2(f_*(X),f_*(W))\nonumber\\
&-\eta_1(X)g_2(f_*(Y),f_*(W))]\leq b_1A\eta_1(R_1(Y,X)W)+B,
\end{eqnarray}
where $b_1=\frac{r}{2nk(2n+1)}$.
\end{thm}\par

We have the conharmonic curvature tensor for a manifold $(M^{2n+1},g)$ given by,
\begin{eqnarray*}
&\Tilde{C}(X,Y)Z=R(X,Y)Z-\frac{1}{(2n-1)}[S(Y,Z)X-S(X,Z)Y\\
&+g(Y,Z)QX-g(X,Z)QY],
\end{eqnarray*}
where $Q$ is the \emph{Ricci operator} and is given by $g(QX,Y)=S(X,Y)$. Now consider the manifold $M_1$ \emph{conharmonically flat} i.e $\Tilde{C}(X,Y)Z=0$. Then using the value of $R_1(X,Y)Z$ from above, we get from the left side inequality of \eqref{e1.2},
\begin{eqnarray*}
&\frac{l_2}{A}[S_1(Y,Z)g_1(X,W)-S_1(X,Z)g_1(Y,W)+g_1(Y,Z)g_1(Q_1X,W)\\
&-g_1(X,Z)g_1(Q_1Y,W)]-B\leq g_2(f_*(R_1(X,Y)Z),f_*(W)).
\end{eqnarray*}
Then putting $Z=\xi_1$ and using \eqref{e2.6}, $g_1(X,\xi)=2nk\eta_1(X)$ and $S_1(Q_1X,Y)=S_1(X,Y)$, we get
\begin{eqnarray*}
&\frac{l_2}{A}[2nk\{\eta_1(Y)g_1(X,W)-\eta_1(X)g_1(Y,W)\}+[\eta_1(Y)S_1(X,W)\\
&-\eta_1(X)S_1(Y,W)]]-B\leq g_2(f_*(R_1(X,Y)\xi_1),f_*(W)).
\end{eqnarray*}
Moreover if $M_1$ is an Einstein manifold, then the above inequality becomes
\begin{eqnarray*}
&\frac{l_2}{A}(2nk+\frac{r}{2n+1})[\eta_1(Y)g_1(X,W)-\eta_1(X)g_1(Y,W)]\\
&-B\leq g_2(f_*(R_1(X,Y)\xi_1),f_*(W)).
\end{eqnarray*}
 Finally, using \eqref{e2.12} and the linearity of $f_*$, the above yields
\begin{equation}\label{e4.22}
\frac{c_1}{A}\eta_1(R_1(Y,X)W)-B\leq k[\eta_1(Y)g_2(f_*(X),f_*(W))-\eta_1(X)g_2(f_*(Y),f_*(W))].
\end{equation}
where, $c_1=\frac{l_2}{k}(2nk+\frac{r}{2n+1})$.\\
Similarly the right inequality of \eqref{e1.2} gives
\begin{equation}\label{e4.23}
k[\eta_1(Y)g_2(f_*(X),f_*(W))-\eta_1(X)g_2(f_*(Y),f_*(W))]\leq c_1A\eta_1(R_1(Y,X)W)+B.
\end{equation}
Therefore from the inequalities \eqref{e4.22} and \eqref{e4.23} we can state
\begin{thm}
Let $M_1(\phi_1,\xi_1,\eta_1,g_1)$ and $M_2(\phi_2,\xi_2,\eta_2,g_2)$ be two $N(k)-$ contact metric manifolds with $\operatorname{dim}M_1=2n+1$ $(n\geq1)$. Suppose $f:M_1\to M_2$ be the quasi-isometric embedding. Furthermore, if the manifold $M_1$ is conharmonically flat Einstein manifold, then for all $X,Y,W\in\chi(M_1)$, the metric $g_2$ of the manifold $M_2$ satisfies:
\begin{eqnarray}\label{e4.24}
\frac{c_1}{A}\eta_1(R_1(Y,X)W)-B\leq k[\eta_1(Y)g_2(f_*(X),f_*(W))\nonumber\\
-\eta_1(X)g_2(f_*(Y),f_*(W))] \leq c_1A\eta_1(R_1(Y,X)W)+B,
\end{eqnarray}
where, $c_1=\frac{l_2}{k}(2nk+\frac{r}{2n+1})=\frac{4n}{2n-1}$, since we have $k=\frac{r}{2n(2n-1)}$ for $N(k)-$ contact Einstein manifolds.
\end{thm}\par

Recall that the Weyl projective curvature tensor $P$ of type $(1,3)$ on a Riemannian manifold $(M^{2n+1},g)$ can be defined as
\begin{equation*}
P(X,Y)Z=R(X,Y)Z-\frac{1}{2n}[S(Y,Z)X-S(X,Z)Y].
\end{equation*}
In a similar calculation, if $M_1$ is \emph{projectively flat} $N(k)-$contact Einstein manifold, i.e. if $P_1=0$ and $S_1(X,Y)=\frac{r}{2n+1}g_1(X,Y)$, then we write
\begin{thm}
Let $M_1(\phi_1,\xi_1,\eta_1,g_1)$ and $M_2(\phi_2,\xi_2,\eta_2,g_2)$ be two quasi- isometrically embedded $N(k)-$contact metric manifolds with $\operatorname{dim}M_1=2n+1)$ $(n\geq1)$. Suppose $fM_1\to M_2$ be such embedding between $M_1$ and $M_2$ with the constants $A\geq1, B\geq0$. Furthermore, if the manifold $M_1$ is projectively flat Einstein manifold, then we have
\begin{eqnarray}\label{e4.25}
\frac{1}{A}\eta_1(R_1(Y,X)W)-B\leq k[\eta_1(Y)g_2(f_*(X),f_*(W))\nonumber\\
-\eta_1(X)g_2(f_*(Y),f_*(W))] \leq A\eta_1(R_1(Y,X)W)+B.
\end{eqnarray}
\end{thm}

\section{Quasi-isometry between two Sasakian manifolds}
Some basic introductory details about the Sasakian manifold is given in the preliminary section. Now we recall an important theorem to establish the rest of the results.
\begin{thm}
\cite{Bla} A $N(k)-$contact metric manifold is Sasakian if and only if $k=1$.
\end{thm}
Using this theorem we can imply the following result from the previous results for $N(k)-$contact metric manifold.
\begin{thm}
Let $M_1(\phi_1,\xi_1,\eta_1,g_1)$ and $M_2(\phi_2,\xi_2,\eta_2,g_2)$ be two Sasakian manifolds with dimension of $M_1=2n+1$ $(n\geq1)$. Let  $f:M_1\to M_2$ be a quasi-isometry embedding between $M_1$ and $M_2$ with constants $A\geq1, B\geq0$. Then the following inequalities hold in the respective following cases:
\begin{enumerate}
    \item If $M_1$ is conformally flat or conformally flat Einstein or concircularly flat or projectively flat manifold, then
        \begin{eqnarray}\label{e5.1}
        \frac{1}{A}\eta_1(R_1(Y,X)W)-B\leq \eta_1(Y)g_2(f_*(X),f_*(W)) \nonumber\\
        -\eta_1(X)g_2(f_*(Y),f_*(W)) \leq A\eta_1(R_1(Y,X)W)+B.
        \end{eqnarray}
    \item If $M_1$ is conharmonically flat Einstein manifold, then
        \begin{eqnarray}\label{e5.2}
        \frac{c_1}{A}\eta_1(R_1(Y,X)W)-B\leq \eta_1(Y)g_2(f_*(X),f_*(W)) \nonumber\\
        -\eta_1(X)g_2(f_*(Y),f_*(W)) \leq c_1 A\eta_1(R_1(Y,X)W)+B,
        \end{eqnarray}
        where $c_1=\frac{4n}{2n-1}$.
\end{enumerate}
\end{thm} \par

\textbf{Example:}
Consider $M_1=\mathbb{R}^3$ with the Euclidean metric $g_1$. Let $\alpha =\frac{1}{2}(dz-ydx)$, $\xi =\frac{\partial}{\partial z}$ and $g_1=\alpha\otimes\alpha +\frac{1}{4}(dx^2+dy^2)$. Then we take $e_1=\frac{\partial}{\partial x}$, $e_2=\frac{\partial}{\partial y}$ and $e_3=\frac{\partial}{\partial z}$ as a set of linearly independent basis vectors for  the set of vector fields $\chi(M_1)$ of the manifold $M_1$. Also consider the $(1,1)$ tensor field $\phi$ be given as, $\phi_1(\frac{\partial}{\partial x})=\frac{\partial}{\partial y}+x\frac{\partial}{\partial z}$, $\phi_1(\frac{\partial}{\partial y})=-\frac{\partial}{\partial x}$ and $\phi_1(\frac{\partial}{\partial z})=0$. Then it can be easily checked that the manifold $(M_1,g_1)$ with the above defined structure is a Sasakian manifold.\par
Take another manifold $M_2=$\{$(x,y,z)\in\mathbb{R}^3: 1< y< 2,z\neq 0$\}, where $(x,y,z)$ are the standard co-ordinates of $\mathbb{R}^3$. Then the linearly independent vector fields are given by $f_1=\frac{\partial}{\partial y}$, $f_2=z^2(\frac{\partial}{\partial z}+2y\frac{\partial}{\partial x})$ and $f_3=\frac{\partial}{\partial x}$. Let $g_2$ be the Riemannian metric defined by: $g_{ij}=1$ for $i=j$ and $g_{ij}=0$ for $i\neq j$. Let $\phi$ be the $(1,1)$ tensor field defined by; $\phi_2 (f_1)=f_3$, $\phi_2 (f_2)=0$ and $\phi_2 (f_3)=-f_1$. Thus for taking $\xi =f_2$, we can show that the manifold $(M_2,g_2)$ with this structure is a Sasakian manifold.\par
Now we define a map $f_*:\chi(M_1)\longrightarrow\chi(M_2)$ on the basis vector fields by,
\begin{equation*}
  f_*(e_1)=\frac{1}{2}(yf_3+\frac{1}{\sqrt{y}}f_1), ~~f_*(e_2)=\frac{1}{2}f_2, ~~f_*(e_3)=-\frac{1}{2}f_3.
\end{equation*}
The $f$ is a quasi-isometry between the two Sasakian manifolds $M_1$ and $M_2$ with the constants $A=2$ and $B=1$.

\vspace{0.2cm}
\section{Quasi-isometric inequality between two Riemannian manifolds}
We will conclude this article with the following result. This theorem concerns between any two Riemannian manifold which have a quasi-isometric structure among them.
\begin{thm}
Let $(M_1, g_1)$ and $(M_2, g_2)$ be two Riemannian manifolds of same dimension $n$ and let $f$ be the quasi-isometric embedding between them with some constants $A\geq1$ and $B\geq0$. Then $$\frac{r_1}{A}-n^2B\leq g_2(f_*(R_1(e_i,e_j)e_j),f_*(e_i))\leq Ar_1+n^2B,$$ $r_1$ being the scalar curvature of the manifold $M_1$.
\end{thm}
\begin{proof}
For all $X$, $Y$, $Z$ and $W$ in $\chi(M_1)$, $R_1(X,Y)Z$ is also in $\chi(M_1)$ and for $f$ being the quasi-isometry between $M_1$ and $M_2$, we get
\begin{eqnarray}\label{e6.1}
 \frac{1}{A}g_1(R_1(X,Y)Z,W)-B \leq &g_2(f_*(R_1(X,Y)Z),f_*(W))\nonumber\\
\leq &Ag_1(R_1(X,Y)Z,W)+B.
\end{eqnarray}
Let $\{e_i\}$ be an orthonormal basis of the tangent space $T_p(M_1)$ at $p\in M_1$. Then for $X=W=e_i$, we get from the left inequality of \eqref{e6.1},
\begin{equation*}
\frac{1}{A}S_1(Y,Z)-nB \leq g_2(f_*(R_1(e_i,Y)Z),f_*(e_i)).
\end{equation*}
Again putting $Y=Z=e_j$ we get
\begin{equation}\label{e6.2}
\frac{1}{A}r-n^2B \leq g_2(f_*(R_1(e_i,e_j)e_j),f_*(e_i)).
\end{equation}
Similarly, right inequality gives
\begin{equation}\label{e6.3}
  g_2(f_*(R_1(e_i,e_j)e_j),f_*(e_i))\leq Ar_1+n^2B.
\end{equation}
 Finally \eqref{e6.2} and \eqref{e6.3} together complete the proof.
\end{proof}

\textbf{Author contributions:} All authors contributed equally to this project.\\

\textbf{Conflicts of Interest:} The authors declare no conflict of interest.

% alteratively, bibliographies prepared with BibTeX can be included by
% means of the following commands
%\bibliographystyle{srtnumbered}
%\bibliography{mybib}

\end{document}